\documentclass[onefignum,onetabnum]{siamart171218}

\usepackage{amssymb}
\usepackage{caption}
\usepackage{hyperref}
\usepackage{subfig}
\newfloat{algorithm}{t}{lop}

\theoremstyle{plain}

\newtheorem{Assumption}{Assumption}
\newtheorem{Model}{Model}

\newcommand\norm[1]{\left\lVert#1\right\rVert}

\usepackage{lipsum}
\usepackage{amsfonts}
\usepackage{graphicx}
\usepackage{epstopdf}
\usepackage{algorithmic}
\ifpdf
  \DeclareGraphicsExtensions{.eps,.pdf,.png,.jpg}
\else
  \DeclareGraphicsExtensions{.eps}
\fi

\newsiamremark{remark}{Remark}
\newsiamremark{hypothesis}{Hypothesis}
\crefname{hypothesis}{Hypothesis}{Hypotheses}
\newsiamthm{claim}{Claim}

\headers{Improved Convergence Bounds For Operator Splitting Algorithms}{A. Hamadouche}

\title{Improved Convergence Bounds For Operator Splitting Algorithms With Rare Extreme Computational Errors\thanks{{Work supported by UK's EPSRC (EP/T026111/1, EP/S000631/1), and the MOD University Defence     Research Collaboration.}}}
\author{Anis Hamadouche\thanks{Anis Hamadouche, Andrew M.\ Wallace, and Jo\~ao F.\ C.\ Mota     are with the School of Engineering \& Physical Sciences, Heriot-Watt University, Edinburgh EH14 4AS,     UK. (e-mail: \{ah225,a.m.wallace,j.mota\}@hw.ac.uk).}
\and Andrew M.\ Wallace\footnotemark[2]
\and Jo\~ao F.\ C.\ Mota\footnotemark[2]}

\usepackage{amsopn}

\makeatletter
\newcommand*{\addFileDependency}[1]{
  \typeout{(#1)}
  \@addtofilelist{#1}
  \IfFileExists{#1}{}{\typeout{No file #1.}}
}
\makeatother

\newcommand*{\myexternaldocument}[1]{%
    \externaldocument{#1}%
    \addFileDependency{#1.tex}%
    \addFileDependency{#1.aux}%
}

\ifpdf
\hypersetup{
  pdftitle={An Example Article},
  pdfauthor={D. Doe, P. T. Frank, and J. E. Smith}
}
\fi

\myexternaldocument{ex_supplement}

\begin{document}

\maketitle

\begin{abstract}
    \textbf{In this paper, we improve upon our previous work~\cite{hamadouche2022sharper, hamadouche2021sspd} and establish convergence bounds on the objective function values of approximate proximal-gradient descent (AxPGD), approximate accelerated proximal-gradient descent (AxAPGD) and approximate proximal ADMM (AxWLM-ADMM) schemes. We consider approximation errors that manifest rare extreme events and we propagate their effects through iterations. We establish probabilistic asymptotic and non-asymptotic convergence\\ bounds  as functions of the range (upper/lower bounds) and variance of approximation errors. We use the derived bound to assess AxPGD in a sparse model predictive control of a spacecraft system and compare its accuracy with previously derived bounds.}
\end{abstract}

\begin{keywords}
  Convex Optimization, Proximal Gradient Descent, Approximate Algorithms; ADMM; Operator Splitting Algorithms
\end{keywords}

\begin{AMS}
  49M37, 65K05, 90C25
\end{AMS}

\section{INTRODUCTION}
Given the following optimization problem
\begin{equation}
\label{gen.prob.1}
\underset{x \in X}{\text{minimize}}\quad f(x) = g(x) + h(x),
\end{equation}
 where $g:\mathbb{R}^n\rightarrow \mathbb{R}$ and $h:\mathbb{R}^n\rightarrow \mathbb{R} \cup \{+\infty\}$ are real-valued convex functions, and the set $X$ is assumed to be closed and convex. $h$ can be non-differentiable in general. Problem \eqref{gen.prob.1} is also known as \textit{structured optimization} problem in the literature.

As was originally proposed in \cite{moreau1965proximite}, problems such as \eqref{gen.prob.1} are easier to solve once approximated by a parametric function $f_\lambda$ through a non-euclidean projection (a.k.a. Moreau envelope) as follows
\begin{equation}
    \label{prox0.1}
     f_\lambda(x) = \underset{z \in \mathbb{R}^n} {\inf} \quad f(z)+ \frac{1}{2\lambda} \norm{z-x}_{2}^2 
\end{equation}
It turns out that $f_\lambda$ is a convex and differentiable function that has the same optimal set $X^\star$ as the original problem \eqref{gen.prob.1} in addition to the fact of being equal to the original function (in value) when both are evaluated at $X^\star$ \cite{moreau1965proximite}. Building on top of this result, \cite{martinet1978perturbation} proposed the \textit{incremental proximal minimization} algorithm that generates a sequence $\{x^k\}$ by solving the following subproblem at every iteration $k$ starting from $x^0$
\begin{equation}
    \label{prox0.2}
     x^{k+1} = \underset{x \in \mathbb{R}^n} {\arg\min} \quad f(x)+ \frac{1}{2\lambda_k} \norm{x-x^{k}}_{2}^2 
\end{equation}
with $\lambda_k > 0$ for all $k \geq 1$. The exact convergence of the latter, and thus its practical usefulness was not known until the seminal work of \cite{rockafellar1976monotone}. Accelerated version of \eqref{prox0.2} was proposed by \cite{guler1992new} in attempt to reproduce the work of \cite{nesterov1983method} in the proximal framework by substituting $x^{k}$ in \eqref{prox0.2} by the following extrapolation step
\begin{equation}
    \label{extrapolate}
    \begin{aligned}
        y^k &= (1-\alpha_k)x^k + \alpha_k \nu^k; \\
        \nu^{k+1} &= \nu^k + \frac{1}{\alpha_k}(x^{k+1}-y^k),
    \end{aligned}
\end{equation}
with some parameter sequence $\{\alpha_k\}$ (see \cite{guler1992new} for more details). Although the proximal term in \eqref{prox0.1} and \eqref{prox0.2} can be any nonquadratic regularization term (i.e, non-euclidean), research efforts thus far have only been focused on proximal methods with quadratic kernels, i.e, $\norm{x-x^k}^2$. In an effort to exploit the geometry of the constraints, alternative entropy-like Bregman kernels were proposed in \cite{chen1993convergence,teboulle1992entropic,eggermont1990multiplicative,eckstein1993nonlinear,censor1992proximal}, whose practical usefulness has been unknown for more than a decade until recently rediscovered in \cite{hanzely2018accelerated,teboulle2018simplified,bolte2018first}. In a more recent work \cite{nesterov2020inexact}, the author suggested to use powers of the Euclidean norm $\norm{.}^{p+1}$ as the regularizing function in \eqref{prox0.2}, which he called the \textit{$p^{th}$-order proximal-point operator} and proved that the latter achieves better convergence rates than classical choices in convex settings.

In problems such as \eqref{gen.prob.1}, because not all $g$ and $h$ are good candidates for the proximal step \eqref{prox0.2} (not proximable), combinations of other techniques with proximal methods for minimizing the sum of two functions such as \eqref{gen.prob.1} was treated in \cite{nesterov2007gradient} but only gained popularity after the famous works of \cite{beck2009fast,beck2009gradient}. The idea stems from replacing the iteration in the \textit{gradient descent} algorithm
\begin{align}
    x^{k+1} = x^k - \lambda_k\nabla f(x^k)
\end{align}
by a gradient step with respect to the differentiable term $g$ followed by an \textit{implicit gradient step} with respect to the non-differentiable term $h$ as follows
\begin{equation}
\begin{aligned}
    \label{grad.iter}
    z^{k+1} &= x^k - \lambda_k\nabla g(x^k)\\
    x^{k+1} &= z^{k+1} - \lambda_k\nabla h(x^{k+1})
\end{aligned}
\end{equation}
or equivalently
\begin{equation}
    \nabla h(x^{k+1}) + \frac{x^{k+1}-(x^k - \lambda_k\nabla g(x^k))}{\lambda_k} = 0,
\end{equation}
which is the necessary optimality condition for the following convex minimization problem
\begin{equation}
    \label{prox.prob.0}
    \underset{x \in \mathbb{R}^n} {\min} \quad h(x) + \frac{\norm{x-(x^k - \lambda_k\nabla g(x^k))}^2}{2\lambda_k}.
\end{equation}
The \textit{minimizer} of \eqref{prox.prob.0} is iteratively given by
\begin{equation}
    \label{proxgrad.iter}
    x^{k+1} =
    \underset{x \in \mathbb{R}^n} {\arg\min} \quad h(x) + \frac{\norm{x-(x^k - \lambda_k\nabla g(x^k))}^2}{2\lambda_k}.
\end{equation}
By incorporating problem constraints into \eqref{grad.iter} and \eqref{proxgrad.iter}, other variants can be obtained using different stepsize sequences in \eqref{grad.iter} and/or use projected gradient steps as follows
\begin{equation}
\begin{aligned}
    \label{grad.iter2}
    z^{k+1} &= P_X(x^k - \lambda_{1k}\nabla g(x^k));\\
    x^{k+1} &= P_X(z^{k+1} - \lambda_{2k}\nabla h(x^{k+1}))
\end{aligned}
\end{equation}
maintaining both sequences $z^k$ and $x^k$ within the constraint set $X$.

Alternatively, \cite{bertsekas2009convex,bertsekas2011incremental} proposed a \textit{proximal-subgradient} algorithm as a variant to \eqref{proxgrad.iter} by taking a sugradient \cite{rockafellar1981theory} step with respect to a general, possibly nondifferentiable, function $g$ in \eqref{grad.iter} rather than a gradient step, thereby relaxing the differentiability and nonincremental property of the proximal gradient method \cite{bertsekas2011incrementalsurvey}.
\begin{remark}
Note that the projection $P_X$ in \eqref{grad.iter2} can be relaxed in either iterations to speed up computation at the expense of potentially inexact solutions.
\end{remark}

In reduced-precision solvers, however, we slightly modify \eqref{proxgrad.iter} by allowing for some perturbations to take place in the gradient step as well in the proximal step to cope with data, hardware and software uncertainties such as \textit{Estimator bias} in approximated gradients \cite{rosasco2019convergence, atchade2014stochastic,nitanda2014stochastic}, \textit{Round-off errors} in \textit{Finite-Precision} arithmetic \cite{higham2002accuracy,higham2020sharper}, \textit{Loop Perforation} \cite{sidiroglou2011managing}, \textit{Voltage Scaling} \cite{hegde1999energy} and \textit{Early Stopping} criteria in approximate computing frameworks \cite{meng2009best}. Incorporating such errors into \eqref{proxgrad.iter} yields the following \textit{approximate proximal-gradient descent} version:
\begin{equation}
\label{AxPGD0}
\begin{aligned}
   x^{k+1} &\in \{x \in C^k_{\epsilon_h}:\frac{1}{2\lambda_k} \norm{x-(x^k - \lambda_{k}\nabla^{\epsilon_g} g(x^k))}_{2}^2+ h(x) \leq \epsilon_h^k \\+ &\underset{z \in \mathbb{R}^n} {\inf} \quad  \frac{1}{2\lambda_k} \norm{z-(x^k - \lambda_{k}\nabla^{\epsilon_g} g(x^k))}_{2}^2+ h(z)\}
\end{aligned}
\end{equation}
with a generic proximity mapping error $\epsilon_h^k$ and a suboptimal set $C^k_{\epsilon_h}$ (a perturbed, $\epsilon_h$-enlarged set~\cite{PertOpBurachik1997enlargement} of the exact proximal set $C^k$). $\nabla^{\epsilon_g^k} g$ stands for the inexact gradient of $g$ calculated at iteration $k$ with some general error $\epsilon_g^k$. Errors $\epsilon_g$ and $\epsilon_h$ are deterministic in nature but since they are not directly measurable they are often considered as random variables for the sake of analysis \cite{higham2002accuracy}.

 If we consider the propagation of errors appearing in \eqref{AxAPGD0} in the extrapolation step of  \eqref{extrapolate} then we obtain the following \textit{approximate accelerated  proximal-gradient descent} version:
\begin{equation}
\label{AxAPGD0}
\begin{aligned}
&
\begin{split}
   x^{k+1} \in \{x \in C^k_{\epsilon_h}:\frac{1}{2\lambda_k} \norm{x-(y^k - \lambda_{k}\nabla^{\epsilon_g} g(y^k))}_{2}^2+ h(x) \leq \epsilon_h^k \\+ \underset{z \in \mathbb{R}^n} {\inf} \quad  \frac{1}{2\lambda_k} \norm{z-(y^k - \lambda_{k}\nabla^{\epsilon_g} g(y^k))}_{2}^2+ h(z)\},
\end{split}\\
&y^k = x^{k} +\beta_k (x^{k}-x^{k-1}),
\end{aligned}
\end{equation}
where the gradient step is taken with respect to $y^k$, which is a non-convex combination of the successive iterates $x^k$ and $x^{k-1}$ with a momentum parameter $\beta_k$ \cite{beck2009fast, nesterov1983method}. Notice that if we set $\beta_k = 0$ in \eqref{AxAPGD0} then we revert back to the basic nonaccelerated scheme \eqref{AxAPGD0}.
%
%

Likewise, another instance that stems from the same family of splitting operator algorithms is the Alternating Direction Method of Multipliers (ADMM)~\cite{peaceman1955numerical, douglas1955numerical, douglas1956numerical, douglas1964general,glowinski1975approximation,gabay1976dual}. This algorithm is typically applied to solve a linearly constrained version of problem~\eqref{gen.prob.1}

\begin{equation}
  \label{Eq:Problem}
  \begin{split}
      &\underset{x \in \mathbb{R}^n,z \in \mathbb{R}^m}{\text{minimize}} \,\,\,
      f(x) := g(x) + h(z)\,,\\
      &\text{subject to} \,\,\,
      Ax+Bz = c\,,  
  \end{split}
\end{equation}
where $g$ and $h$ are possibly convex and nondifferentiable functions, $A \in \mathbb{R}^{p\times n}$ and $B \in \mathbb{R}^{p\times m}$. ADMM-type algorithms have a wide spectrum of applications in machine learning \cite{dhar2015admm,ding2019stochastic,liu2020admm, holmes2021nxmtransformer}, artificial intelligence \cite{yang2018admm,li2019admm}, MIMO detection \cite{un2019deep}, image reconstruction \cite{yang2018admm}, compressed sensing \cite{mota2013d, mota2011distributed, un2019deep} and model predictive control \cite{mota2014distributed, zhang2017embedded,darup2019towards,peccin2019fast, cheng2020semi, rey2016admm, tang2019distributed}.

To apply ADMM to \eqref{Eq:Problem}, we first form the \textit{augmented Lagrangian}
\begin{equation}
\label{Lagrangian}
\begin{split}
    \mathcal{L}_\rho(x,z, y) =& g(x)+h(z) + y^\top (Ax+Bz - c) 
    + \frac{\rho}{2}\norm{Ax+Bz-c}_2^2     
\end{split}
\end{equation}
where $y\in \mathbb{R}^p
$ is the Lagrange multiplier associated with the linear constraint in \eqref{Eq:Problem}, and $\rho > 0$ is the dual step-size parameter. ADMM is then obtained by sequentially minimizing the augmented Lagrangian \eqref{Lagrangian} with respect to $x$ and $z$, and then updating $y$, as follows
\begin{subequations}
\label{ADMM0}
    \begin{alignat}{3}
    &x^{k+1} = \underset{x}{\arg\min}\quad \mathcal{L}_\rho(x,z^k, y^k)&\\
    &z^{k+1} = \underset{z}{\arg\min}\quad \mathcal{L}_\rho(x^{k+1},z, y^k)&\\
    &y^{k+1} = y^k+\rho(Ax^{k+1}+Bz^{k+1}-c).&
    \end{alignat}
\end{subequations}
The equivalent \textit{scaled proximal ADMM} is given by 
\begin{subequations}
\label{ADMMProx}
    \begin{alignat}{3}
    &x^{k+1} = \underset{x}{\arg\min}\quad g(x)+\frac{1}{2\lambda}\|Ax+Bz^k-c+v^k\|_2^2&\label{ADMMProxSub1}\\
    &z^{k+1} = \underset{z}{\arg\min}\quad h(z)+\frac{1}{2\lambda}\|Ax^{k+1}+Bz-c+v^k\|_2^2&\label{ADMMProxSub2}\\
    &v^{k+1} = v^k+(Ax^{k+1}+Bz^{k+1}-c).&
    \end{alignat}
\end{subequations}
where $v^k=({1}/{\rho})y^k$ and $\lambda = 1/\rho$.

By adding a quadratic function in \eqref{ADMMProxSub1} and \eqref{ADMMProxSub2} penalizing the difference between the variable and its previous value, a term known as \textit{proximal penalty}, yields
\begin{subequations}
\label{WLM-ADMM}
    \begin{alignat}{3}
    &x^{k+1} = \underset{x}{\arg\min}\quad g(x)+\frac{1}{2\lambda}\|Ax+Bz^k-c+v^k\|_2^2+\frac{\lambda_x}{2}\|x-x^k\|_2^2&\\
    &z^{k+1} = \underset{z}{\arg\min}\quad h(z)+\frac{1}{2\lambda}\|Ax^{k+1}+Bz-c+v^k\|_2^2+\frac{\lambda_z}{2}\|z-z^k\|_2^2&\\
    &v^{k+1} = v^k+(Ax^{k+1}+Bz^{k+1}-c),&
    \end{alignat}
\end{subequations}
where $\lambda_x^k$ and $\lambda_z^k$ are positive scalars that can possibly vary from iteration to iteration. Defining
\begin{subequations}
\label{WLM-ADMMProx2Norm}
    \begin{alignat}{3}
    & \gamma_{1_k} = \lambda_x  x^k-\frac{\lambda_x}{\lambda}A^\top  \bigg(Bz^k-c+v^k\bigg)&\\
    & \gamma_{2_k} =  \lambda_z  z^k-\frac{\lambda_z}{\lambda}B^\top  \bigg(Ax^{k+1}-c+v^k\bigg)&\\  
&\mathcal{M}_{\frac{1}{\lambda_x} g}(\gamma_{1_k}) = \underset{x}{\inf}\quad g(x)+\frac{\lambda_x}{2}\|x -\gamma_{1_k}\|_2^2&\\
&\mathcal{M}_{\frac{1}{\lambda_z} h}(\gamma_{2_k}) = \underset{z}{\inf}\quad h(z)+\frac{\lambda_z}{2}\|z -\gamma_{2_k}\|_2^2&,
    \end{alignat}
\end{subequations}

\eqref{WLM-ADMM} can be rewritten as
\begin{subequations}
\label{WLM-ADMMProx}
    \begin{alignat}{3}
    &x^{k+1} =  \gamma_{1_k} - \frac{1}{\lambda_x} \nabla \mathcal{M}_{\frac{1}{\lambda_x} g}(\gamma_{1_k})&\\
    &z^{k+1} = \gamma_{2_k} - \frac{1}{\lambda_z} \nabla \mathcal{M}_{\frac{1}{\lambda_z} h}(\gamma_{2_k})&\\
    &v^{k+1} = v^k+A\gamma_{1_k}+B\gamma_{2_k}-c- \frac{1}{\lambda_x}A\nabla \mathcal{M}_{\frac{1}{\lambda_x} g}(\gamma_{1_k})-\frac{1}{\lambda_z}B\nabla \mathcal{M}_{\frac{1}{\lambda_z} h}(\gamma_{2_k})&.
    \end{alignat}
\end{subequations}


Incorporating approximate computational errors into \eqref{WLM-ADMMProx} yields the following \textit{approximate ADMM} version:
\begin{equation}
\begin{aligned}
   x^{k+1} &\in \bigg\{x \in C^k_{\epsilon_g}:g(x)+\frac{\lambda_x}{2}\|x -\gamma_{1_k}\|_2^2 \leq \epsilon_g^k+\underset{y \in \mathbb{R}^n} {\inf} \quad  g(y)+\frac{\lambda_x}{2}\|y -\gamma_{1_k}\|_2^2\bigg\}\\
   z^{k+1} &\in \bigg\{z \in C^k_{\epsilon_h}:h(z)+\frac{\lambda_z}{2}\|z -\gamma_{2_k}\|_2^2 \leq \epsilon_h^k+\underset{w \in \mathbb{R}^n} {\inf} \quad  h(w)+\frac{\lambda_z}{2}\|w -\gamma_{2_k}\|_2^2\bigg\}\\
    v^{k+1} &= v^k+Ax^{k+1}+Bz^{k+1}-c,
\end{aligned}
\end{equation}
with a generic proximal computation error $\epsilon_{(\cdot)}$ and its associated suboptimal set $C^k_{\epsilon_{(\cdot)}}$.

Since approximation errors can adversely affect arithmetic operations in approximate computing designs \cite{ragavan2017pushing}, we need a verification framework whereby the propagation of such errors throughout the algorithm in question is better understood. This would consequently help us determine the extent to which a particular algorithm can tolerate hardware and software uncertainties, hence formulate less-conservative assertions that can be used in application's performance vs computation efficiency trade-offs. 
\subsection{Our approach}
We propose a general class of algorithms that stems from ADMM but with weighted Lagrangian and generalized proximal terms. We prove that most existing proximal-based optimisation algorithms are instances of this class. We then adopt an analytical approach and we propagate the approximation errors through iterations. We analyse the error contributing terms, and assuming some rare extreme event approximation error models, we apply suitable concentration inequalities to eliminate the dependency on the iteration counter and establish \textit{a priori} probabilistic error bounds that hold with some probability which we quantify with a lower bound. We apply the same line of proof to proximal-gradient, accelerated proximal-gradient as well as the weighted-lagrangian generalized proximal ADMM. Assuming subgaussian error distributions, instead of subexponential, makes it possible to exploit the second-order statistical information of approximation errors to improve the bounds using a set of concentration inequalities that are best-suited to light-tailed error models.
\subsection{Applications}
 In applications that tolerate some loss in accuracy, we can use reduced precision computations such as mixed-precision fixed-point arithmetic or any custom representation to promote some gain in resource utilization or power efficiency. In test-driven development of complex algorithms, we typically require more thoughtful and accurate assertions that can accurately predict the application performance (convergence in our case), and hence leads to more efficient and less conservative designs. Probabilistic assertions (convergence bounds) are a powerful tool when it comes to ML and AI applications which do not require convergence to hold in all cases. In a previous work, we noticed that round-off errors manifest some light-tailed rare extreme events behaviour. This work builds on top of the latter to derive probabilistic bounds (assertions) for convergence of approximate first-order solvers when running on reduced precision reconfigurable computing machines such as FPGAs. In this work we mimic embedded hardware approximation errors by injecting random noise with high range to variance ratio to proximal-gradient algorithm iterations when used as a solver in model predictive control of a spacecraft system. 
\subsection{Contributions}
We improve upon our previous work~\cite{hamadouche2021sspd, hamadouche2022sharper} and establish convergence bounds on the objective function values of approximate proximal-gradient descent (AxPGD), approximate accelerated proximal-gradient descent (AxAPGD) and the generalized Proximal ADMM (AxWLM-ADMM) schemes. We consider approximate gradient and approximate proximal computations with errors that manifest light-tailed rare extreme events behaviour and quantify their contribution to the final residual suboptimality in the objective function value. We establish probabilistic asymptotic and non-asymptotic convergence bounds that can be used as probabilistic assertions that only depend on the range and variance of approximation errors. We apply the derived bound to evaluate the basic proximal-gradient algorithm in a sparse model predictive control problem and compare its sharpness with previously derived bounds. 
\subsection{Organization}
This paper is organized as follows. In Section~\ref{Section:Related Work} we present the most recent convergence bounds that are mostly related to the current work. In Section~\ref{Section:Main Results}, we propose a general class of operator splitting algorithms (WLM-ADMM), from which we instantiate the classical proximal-gradient descent algorithm (PGD) the accelerated proximal-gradient descent algorithm (APGD) and their approximate variants (AxWLM-ADMM, AxPGD and AxAPGD). We then state the main assumptions and error models which are used in the main results, which are also proved within the same section. We validate our proposed approach with an MPC application in Section~\ref{Section: Experimental Results} and we conclude the paper by a summary and possible future directions and extensions in Section~\ref{Section: Conclusions}.
\section{RELATED WORK}
\label{Section:Related Work}
The work in \cite{schmidt2011convergence} establishes the following ergodic convergence bound in terms of function values of the averaged iterates for the basic approximate proximal-gradient~\eqref{AxPGD0} descent algorithm
  \begin{equation}
    \label{schmidt1}
    \begin{split}
    &f\bigg(\frac{1}{k}\sum_{i=1}^{k}x^{i}\bigg)  
    -
    f(x^\star) 
    \leq 
    \frac{L}{2k}
    \Big[\norm{x^\star-x^0}_{2}+ 2A_k + \sqrt{2B_k}\Big]^2\\
    &\quad A_k = \sum_{i=1}^{k}\Big(\frac{\|\epsilon_g^{i}\|_2}{L}+\sqrt{\frac{2\epsilon_h^{i}}{L}}\Big), \quad B_k = \sum_{i=1}^{k} \frac{\epsilon_h^{i}}{L},
    \end{split}
  \end{equation}
where $x^\star$ is an optimal solution of \eqref{Eq:Problem}, $L$ is the \textit{Lipschitz} constant of the gradient, and $x^0$ is the initialization vector. The same work also analyzed the approximate accelerated proximal gradient~\eqref{AxAPGD0} and obtained the following convergence result in terms of the function values of the iterates,
  \begin{equation}
    \label{schmidt2}
    \begin{split}
    &f\big(x^{i}\big)  
    -
    f(x^\star) 
    \leq 
    \frac{2L}{(k+1)^2}
    \Big[\norm{x^\star-x^0}_{2}+ 2\tilde{A}_k + \sqrt{2\tilde{B}_k}\Big]^2\\
    &\quad \tilde{A}_k = \sum_{i=1}^{k} i \Big(\frac{\|\epsilon_g^{i}\|_2}{L}+\sqrt{\frac{2\epsilon_h^{i}}{L}}\Big), \quad \tilde{B}_k = \sum_{i=1}^{k} \frac{i^2\epsilon_h^{i}}{L}.
    \end{split}
  \end{equation}
Assuming random approximation error sequences, $\epsilon_{g_\Omega}^{k}$ and $\epsilon_{h_\Omega}^{k}$, that are drawn from some sample space $\Omega$ of a given probability measure $\mathbb{P}$, in our previous work~\cite{hamadouche2022sharper}, we derived sharper probabilistic convergence error bounds that do not depend on the iteration counter $k$, which makes them more realistic assertions for testing convergence. The bounds that we derived are given below for the basic approximate proximal-gradient descent~\eqref{AxPGD0}
\begin{equation}
    f\bigg(\frac{1}{k}\sum_{i=1}^{k}x_{_\Omega}^{i}  \bigg)-f(x^\star)
\lessapprox \frac{1}{k}\sum_{i=1}^{k}\epsilon_{g_\Omega}^{i} + \gamma M_{\nabla g}{D_x}\sqrt{\frac{n}{k}}|\delta|\norm{x^\star-x^0}_{2}+
    \frac{{D_x^2}}{2sk}\norm{x^\star-x^0}_{2}^2
    \label{Eq:2.3}.
\end{equation}
And for the accelerated proximal-gradient descent~\eqref{AxAPGD0} we obtained the following 
\begin{equation}
   f(x_{_\Omega}^{k+1})-f(x^\star) \leq \frac{1}{\alpha_k^2} \bigg[S_{\epsilon_{g_\Omega}}+S_{r_{_\Omega}}+S_{\epsilon_{h_\Omega}} +\frac{1}{2s}\norm{x^\star-x^0}_{2}^2\bigg]
   \label{Eq:2.4},
\end{equation}
where 
\begin{align}
    S_{\epsilon_{h_\Omega}}&={\varepsilon_0\sum_{i=0}^{k}{i}^2} + \frac{\gamma}{2}\sqrt{\sum_{i=1}^{k}i^{4}(\epsilon_{h_\Omega}^i)^2},\\
    S_{\epsilon_{g_\Omega}}&=\gamma|\delta|M_{\nabla g}{D_u^2 \norm{x^0-x^\star}_2^2\sqrt{n\sum_{i=1}^{k} i^2}},\\
    S_{r_{_\Omega}}&=\gamma{D_u^2 \norm{x^0-x^\star}_2^2\sqrt{\frac{2}{s}\sum_{i=1}^{k} i^2\epsilon_h^i}}
\end{align}
with probability at least $\left(1-{4}\exp(-\gamma^2/2)\right)$, where $x^\star$ is any solution of~\eqref{gen.prob.1}, $M_{\nabla g} = \underset{i \in \mathbb{N}_{+}} \sup\bigg\{\norm{\nabla g(x^i)}_{\infty}\bigg\}$, $\varepsilon_0$ and $\delta$ are upper bounds on the errors, $s$ is the stepsize, $D_x$, $D_u$ and $\gamma$ are some scalar parameters that depend on the probability of the assertion, and $\Omega$ stands for the sample space. The summation concentration bounds depend on the range $[a_i, b_i]$ of the individual random variables. Often, the variance $\sigma_i^2$ of each random variable is significantly smaller than the range. In such situations, the Bernstein inequality provides useful bounds.


Assuming subexponential distribution of the approximation errors, the previously derived probabilistic bounds are based on Hoeffding concentration inequality under boundedness and conditional mean independence (CMI) assumptions on the errors~\cite{hamadouche2022sharper, hamadouche2021sspd}. However, having analysed fixed-point arithmetic computational errors, we have noticed that these errors manifest some light-tail rare extreme events behaviour. 
\section{MAIN RESULTS}
\label{Section:Main Results}
Before we present our new results, let us state the main class of algorithms that we will consider during the convergence analysis.
\subsection{Setup and algorithms}
Defining
\begin{subequations}
    \begin{alignat}{3}
& \Lambda_{1_k} = \frac{1}{\lambda} A^\top L_k  A + M_x^k,&\\
& \gamma_{1_k}(x^k,z^k,v^k) = M_x^k x^k-\frac{1}{\lambda}A^\top L_k  (Bz^k-c+v^k),&\\
& \Lambda_{2_k} = \frac{1}{\lambda} B^\top L_k  B + M_z^k&\\
& \gamma_{2_k}(x^{k+1},z^k,v^k) = M_z^k z^k-\frac{1}{\lambda}B^\top L_k  (Ax^{k+1}-c+v^k).&    
    \end{alignat}
\label{Definition:LambdaGamma}
\end{subequations}

Then the generalized Moreau envelope can be defined as
\begin{subequations}
\begin{alignat}{3}
&\mathcal{M}_{\frac{1}{\lambda_x} g}(\gamma_{1_k}) = \underset{x}{\inf}\quad \frac{1}{\lambda_x} g(x)+\frac{1}{2}\|x -\Lambda_{1_k}^{-1}\gamma_{1_k}\|_{\Lambda_{1_k}}^2,&\\
&\mathcal{M}_{\frac{1}{\lambda_z} h}(\gamma_{2_k}) = \underset{z}{\inf}\quad \frac{1}{\lambda_z} h(z)+\frac{1}{2}\|z -\Lambda_{2_k}^{-1}\gamma_{2_k}\|_{\Lambda_{2_k}}^2&,
\end{alignat}
\end{subequations}
and the Weighted-Lagrangian Generalized-Proximal ADMM (WLM-ADMM)~\cite{hamadouche2022probabilistic} iterations are computed as
\begin{subequations}
    \begin{alignat}{3}
    &x^{k+1} =  \gamma_{1_k} - \frac{1}{\lambda_x} \nabla \mathcal{M}_{\frac{1}{\lambda_x} g}(\gamma_{1_k}),&\\
    &z^{k+1} = \gamma_{2_k} - \frac{1}{\lambda_z} \nabla \mathcal{M}_{\frac{1}{\lambda_z} h}(\gamma_{2_k}),&\\
    &v^{k+1} = v^k+A\gamma_{1_k}+B\gamma_{2_k}-c- \frac{1}{\lambda_x}A\nabla \mathcal{M}_{\frac{1}{\lambda_x} g}(\gamma_{1_k})-\frac{1}{\lambda_z}B\nabla \mathcal{M}_{\frac{1}{\lambda_z} h}(\gamma_{2_k})&.
    \end{alignat}
\end{subequations}
If we consider approximate computational errors in \eqref{WLM-ADMMProx} then we obtain the following approximate (AxWLM-ADMM) algorithm
\begin{subequations}
\label{AxWLM-ADMMProx}
\begin{align}
   \begin{split}
    x^{k+1} &\in \bigg\{x \in C^k_{\epsilon_g}:g(x)+\frac{\lambda_x}{2}\|z -\Lambda_{1_k}^{-1}\gamma_{1_k}\|_{\Lambda_{1_k}}^2 \leq \epsilon_g^k+\underset{y \in \mathbb{R}^n} {\inf} g(y)\\&\quad+\frac{\lambda_x}{2}\|y -\Lambda_{1_k}^{-1}\gamma_{1_k}\|_{\Lambda_{1_k}}^2\bigg\},
   \end{split}\\
   \begin{split}
    z^{k+1} &\in \bigg\{z \in C^k_{\epsilon_h}:h(z)+\frac{\lambda_z}{2}\|z -\Lambda_{2_k}^{-1}\gamma_{2_k}\|_{\Lambda_{2_k}}^2 \leq \epsilon_h^k+\underset{w \in \mathbb{R}^n} {\inf} h(w)\\&\quad+\frac{\lambda_z}{2}\|w -\Lambda_{2_k}^{-1}\gamma_{2_k}\|_{\Lambda_{2_k}}^2\bigg\},
   \end{split}\\
   v^{k+1} &= v^k+A\gamma_{1_k}+B\gamma_{2_k}-c- \frac{1}{\lambda_x}A\nabla \mathcal{M}_{\frac{1}{\lambda_x} g}(\gamma_{1_k})-\frac{1}{\lambda_z}B\nabla \mathcal{M}_{\frac{1}{\lambda_z} h}(\gamma_{2_k}),
\end{align}
\end{subequations}
where $\epsilon_g$ and $\epsilon_h$ are errors with given models and assumptions. 

For unconstrained problems such as~\eqref{gen.prob.1}, if we choose $M_x = 0$ and $M_z = I$ in~\eqref{Definition:LambdaGamma}, then we obtain the following approximate proximal-gradient descent algorithm (AxPGD)
\begin{subequations}
\label{AxPGD}
    \begin{align}
    x^{k+0.5} &=  x^k - \frac{1}{\lambda_x} \nabla g^{\epsilon_g^k}(x^k),\label{AxPGD-x-update}\\
    \begin{split}
    x^{k+1} &\in \bigg\{z \in C^k_{\epsilon_h}:h(z)+\frac{\lambda_z}{2}\|z -x^{k+0.5}\|_2^2 \leq\epsilon_h^k+\underset{w \in \mathbb{R}^n} {\inf} h(w)\\ &\quad+\frac{\lambda_z}{2}\|w -x^{k+0.5}\|_2^2\bigg\},
    \end{split}\\
    &=\operatorname{prox}_{s h}^{\epsilon_{h}^{k}}     \Big[x^{k} - s\big(\nabla g(x^{k})+\epsilon_g^{k}\big)\Big]\,.
    \end{align}
\end{subequations}
where $s = \frac{1}{\lambda_x}$ and the inexact $\epsilon_g^k$-gradient $\nabla g^{\epsilon_g^k}(x^k)$ is defined as
\begin{equation}
    \nabla g^{\epsilon_g^k}(x^k) := \nabla g (x^k)+ \epsilon_g^k.
\end{equation}
Using the extrapolation step $y^k = x^k +\beta_k(x^k-x^{k-1})$, with some positive scalar (momentum) sequence $\{\beta_k\}$, instead of $x^k$ in~\eqref{AxPGD-x-update} we obtain the approximate accelerated proximal-gradient descent (AxAPGD)
    \begin{subequations}
        \label{AxAPGD}
        \begin{align}
        y^k &= x^k +\beta_k(x^k-x^{k-1}),\\
        x^{k+0.5} &=  y^k - \frac{1}{\lambda_x} \nabla g^{\epsilon_g^k}(y^k),\\
        \begin{split}
        x^{k+1} &\in \bigg\{z \in C^k_{\epsilon_h}:h(z)+\frac{\lambda_z}{2}\|z -x^{k+0.5}\|_2^2 \leq\epsilon_h^k+\underset{w \in \mathbb{R}^n} {\inf} h(w)\\ &\quad+\frac{\lambda_z}{2}\|w -x^{k+0.5}\|_2^2\bigg\},
        \end{split}\\
        &=\operatorname{prox}_{s h}^{\epsilon_{h}^{k}}     \Big[y^{k} - s\big(\nabla g(y^{k})+\epsilon_g^{k}\big)\Big]\,.
        \end{align}
    \end{subequations}
\subsection{Problem assumptions and error models}
Let us now state the main assumptions and error models that we will adopt during the analysis.
\begin{Assumption}[Assumptions on the problem]
  \label{Ass:OptimizationProblem}
  \hfill    
            
  \medskip  
  \noindent
  \begin{itemize}

    \item The function $h:\mathbb{R}^n \to \mathbb{R} \cup \{+\infty\}$ is
      closed, proper, and convex.

    \item The function $g\,:\, \mathbb{R}^{n} \to \mathbb{R}$ is convex and
      differentiable,
      and its gradient $\nabla g\,:\, \mathbb{R}^n \to \mathbb{R}^n$ is
      Lipschitz-continuous with constant $L > 0$, that is,
      \begin{equation}
        \label{Eq:LipschitzContinuity}
        \big\|\nabla g(y) - \nabla g(x)\big\|_2 \leq L\big\|y - x\big\|_2\,,
      \end{equation}
      for all $x$, $y \in \mathbb{R}^n$, where $\|\cdot\|_2$ stands for the
      standard Euclidean norm. 

    \item The set of optimal solutions of~\eqref{gen.prob.1} is nonempty: 
      \begin{equation}
        \label{Eq:NonemptySetSolutions}
        X^\star := \big\{x \in \mathbb{R}^n\, :\, f(x) \leq
        f(z),\,\,\text{\emph{ for all} $z \in \mathbb{R}^n$}\big\} \neq
        \emptyset\,.
      \end{equation}
  \end{itemize}
\end{Assumption}
Let us now introduce the main concentration inequality that will be applied to the error terms.

\begin{lemma}[Bernstein Inequality~\cite{wainwright2019high}]
Suppose that the random variables $X_i$, $i= 1,\dots,n$ are
independent and sub-Gaussian satisfying the following:
\begin{itemize}
    \item there exists a constant $M > 0$ such that $\mathbb{P}\left(|X_i-\mathbb{E}(X_i)|\leq M\right) = 1$ for each $i \in [1, n]$, and
    \item the variance of each $X_i$ is $\sigma_i^2$.
\end{itemize} 
If we define $S = \sum_{i=1}^{k}X_i$ then for all $t \geq 0$, we have
\begin{equation} 
    \label{Lemma:BernsteinBound}
    \mathbb{P}\bigg(|S - \mathbb{E}\big[S\big]|\geq t\bigg)\leq 2\textup{exp}\bigg(\frac{-t^2}{2\sum_{i=1}^n \sigma_i^2 + \frac{2}{3}Mt}\bigg). 
\end{equation}
\end{lemma}

In order to apply these bounds to the approximation errors in the proximal gradient descent algorithm, we need to define new models for the gradient and proximal errors. The types of errors that our analysis covers, i.e., random errors with rare extreme events are introduced next.

\begin{Model}
\label{ErrModel3}
Under this model, the approximate gradient computations are subject to additive random error $\epsilon_{g_\Omega}$ whose magnitude is bounded and whose variance is known. Specifically, the evaluation of the gradient of $g$ in~\eqref{gen.prob.1} is approximated as
\begin{equation}
    \label{Eq:Model1}
    \nabla g^{\epsilon_{g_\Omega}}(x^k) = \nabla g (x^k)+ \epsilon_{g_\Omega}^k,
\end{equation}
with  $\epsilon_{g_\Omega}^k$, for $k \geq 1$, satisfying:
\begin{enumerate}
\item $\epsilon_{g_\Omega}^k$ is conditionally mean independent (CMI) of past iterates
\begin{equation}
    \mathbb{E}\big[
    \epsilon_{g_\Omega}^k
    \,\big\vert\,
    \epsilon_{g_\Omega}^1, \ldots, \epsilon_{g_\Omega}^{k-1}
    \big]=\mathbb{E}\big[\epsilon_{g_\Omega}^k\big].\,
\end{equation}
\item $\epsilon_{g_\Omega}^k$ is zero-mean 
\begin{equation}
\mathbb{E}\big[\epsilon_{g_\Omega}^k\big]
    =
    0.
\end{equation}
\item $\epsilon_{g_\Omega}^k$ is bounded \begin{equation}
    \mathbb{P}\left(\left|\left(\epsilon_{g_\Omega}^k\right)_j \right| \leq M_{\epsilon_{g}}\right) = 1\,,\quad
    \text{\emph{for all} $j = 1, \ldots, n$}
\end{equation}
\item $\epsilon_{g_\Omega}^k$ is data mean independent (DMI)
\begin{equation}
    \mathbb{E}\big[
    {\epsilon_{g_\Omega}^k }^\top x_{_\Omega}^k 
    \,\big\vert\,
    \epsilon_{g_\Omega}^1, \ldots, \epsilon_{g_\Omega}^{k-1},\,
    x_{1_\Omega}^1, \ldots, x_{1_\Omega}^{k-1}
    \big]
    =
    \mathbb{E}\big[
    {\epsilon_{g_\Omega}^k }^\top x_{_\Omega}^k 
    \big]
    =
    0.
\end{equation}
\item For each entry $\left(\epsilon_{g_\Omega}^k\right)_j$, the variance is significantly smaller than the upper bound 
\begin{equation}
    \mathbb{E}\left[(\epsilon_{g_\Omega}^k )_j^2\right] = \sigma_j^2 \ll M_{\epsilon_{g}},
\end{equation}
where  $\left(\epsilon_{g_\Omega}^k\right)_j$ denotes
the $j$-th entry of $\epsilon_{g_\Omega}^k$, and 
\begin{equation}
M_{\epsilon_{g}} = \begin{cases}
\delta \cdot \underset{i \in \mathbb{N}_{+}} \sup\bigg\{\norm{\nabla g(x^i)}_{\infty}\bigg\}, & \text{If relative error is used},\\
\delta, & \text{If absolute error is used},
 \end{cases}
 \end{equation} 
where $\delta$ is the machine precision.
\end{enumerate}
\end{Model}

\begin{Model}
\label{ErrModel4}
The inexact proximal operator $\operatorname{prox}_{\cdot}^{\epsilon_{\cdot_\Omega}}$ is subject to random approximation error $\epsilon_{\cdot_\Omega}^k$, or equivalently, to additive random error $r_{\cdot_\Omega}^k$ (in the image of the proximal operator) that satisfy the following:
\begin{enumerate}
\item $r_{_\Omega}^k$ is CMI:
\begin{equation}
      \mathbb{E}\big[
      r_{\cdot_\Omega}^k
      \,\big\vert\,
      r_{\cdot_\Omega}^1, \ldots, r_{\cdot_\Omega}^{k-1}
      \big]
      =
      \mathbb{E}\big[r_{\cdot_\Omega}^k\big]
      =
      0\,,
      \label{Eq:ResidAssumptionRandomZeroMean}
\end{equation}
\item $r_{\cdot_\Omega}^k$ is DMI:
\begin{equation}
      \mathbb{E}\big[
      {r_{\cdot_\Omega}^k }^\top x_{_\Omega}^k 
      \,\big\vert\,
      r_{\cdot_\Omega}^1, \ldots, r_{\cdot_\Omega}^{k-1},\,
      x_{1_\Omega}^1, \ldots, x_{1_\Omega}^{k-1}
      \big]
      =
      \mathbb{E}\big[
      {r_{\cdot_\Omega}^k }^\top x_{_\Omega}^k 
      \big]
      =
      0.
\end{equation}
\item
For $k\geq 1$, the proximal error $\epsilon_{\cdot_\Omega}^k$ is bounded almost surely. Specifically, for some $\varepsilon_{\cdot 0} > 0$, with probability $1$, we have that
\begin{equation}
    0 \leq \epsilon_{\cdot_\Omega}^k \leq \varepsilon_{\cdot_0}.
\end{equation}
\item
$\epsilon_{_\Omega}$ is stationary. Specifically, for all $k$, we have 
\begin{equation}
    \mathbb{E}\bigg[\epsilon_{\cdot_\Omega}^k\bigg]=\mathbb{E}\bigg[\epsilon_{\cdot_\Omega}\bigg]=\text{const.}
\end{equation}
\item The variance of $\epsilon_{\cdot_\Omega}^k$ is significantly smaller than the upper bound 
\begin{equation}
    \mathbb{E}\left[(\epsilon_{\cdot_\Omega}^k )^2\right] = \sigma_\cdot^2 \ll {\varepsilon_{\cdot 0}},
\end{equation}
\end{enumerate}
\end{Model}

\subsection{Convergence bounds for the approximate basic PGD}
Consider solving problem~\eqref{gen.prob.1} using~\eqref{AxPGD} under error models~\ref{ErrModel3} and~\ref{ErrModel4}.
\begin{theorem}[\textbf{Convergence bounds for AxPGD under approximation errors with rare extreme events}]
\label{Theorem3.2}
Consider problem~\eqref{gen.prob.1} and let Assumption~\ref{Ass:OptimizationProblem} hold. Assume that the gradient error $\{\epsilon_{g_\Omega}^k\}_{k\geq 1}$ and residual proximal error $\{r_{_\Omega}^{k}\}_{k\geq 1}$ sequences satisfy the properties of Model~\ref{ErrModel3} and Model~\ref{ErrModel4}, respectively. Let $\{x_{\Omega}^i\}$ denote a sequence generated by the approximate proximal-gradient descent algorithm in~\eqref{AxPGD} with constant stepsize $s_k = s$, for all $k$. Assume that $\norm{x_{_\Omega}^k-x_{_\Omega}^\star}_{2}^2 \leq \norm{x_{_\Omega}^0-x_{_\Omega}^\star}_{2}^2$ hold with probability $p$, for all $k$. Then, for any $\gamma >0$,
\begin{equation}
   f(x_{_\Omega}^{k+1})-f(x^\star) \leq \frac{1}{k} \bigg[S_{\epsilon_{h_\Omega}}+S_{r_{_\Omega}}+S_{\epsilon_{g_\Omega}} +\frac{1}{2s}\norm{x^\star-x^0}_{2}^2\bigg]
\end{equation}
where 
\begin{align}
    S_{\epsilon_{h_\Omega}}&=k\mathbb{E}\big[\epsilon_{h_\Omega}\big]+\gamma\sqrt{(\varepsilon_0-\mathbb{E}\big[\epsilon_{h_\Omega}\big])\mu_1 k}\\
    &=k \mathbb{E}\big[\epsilon_{h_\Omega}\big]+\gamma\sqrt{k\cdot \underset{i \in \mathbb{N}_{+}} \sup\left\{\sigma_{\epsilon_h}^2\right\}},\quad \text{for}\quad k \gg \frac{\gamma^2 (\varepsilon_0-\mathbb{E}\big[\epsilon_{h_\Omega}\big])^2}{9{\underset{i \in \mathbb{N}_{+}} \sup\left\{\sigma_{\epsilon_h}^2\right\}}},\\
    S_{\epsilon_{g_\Omega}}&=\mathbb{E}\big[\sum_{i=0}^k {\epsilon_{g_\Omega}^{i}}^\top (x^\star - x_{_\Omega}^{i+1})\big]+\gamma\sqrt{\mu_2 k\sqrt{n}\delta M_{\nabla g}\norm{x_{_\Omega}^\star-x_{_\Omega}^0}_{2}},\\
    &\approx\gamma\sqrt{k\sqrt{n}\delta M_{\nabla g}\norm{x_{_\Omega}^\star-x_{_\Omega}^0}_{2}},\\
    &=\gamma\sqrt{k\cdot {\underset{i \in \mathbb{N}_{+}} \sup\left\{\sigma_{\epsilon_g,i}^2\right\}}},\quad \text{for}\quad k \gg \frac{\gamma^2 n(\delta M_{\nabla g}\norm{x_{_\Omega}^\star-x_{_\Omega}^0}_{2})^2}{9{\underset{i \in \mathbb{N}_{+}} \sup\left\{\sigma_{\epsilon_g,i}^2\right\}}},\\
    S_{r_{_\Omega}}&=\mathbb{E}\big[-\sum_{i=0}^k \frac{1}{s}{r_{_\Omega}^{i+1}}^\top (x^\star - x_{_\Omega}^{i+1})\big]+\gamma\sqrt{\mu_3 k\sqrt{\frac{2\epsilon_h^k}{s}}\norm{x_{_\Omega}^\star-x_{_\Omega}^0}_{2}},\\
    &\lessapprox\gamma\sqrt{\mu_3 k\sqrt{\frac{2\varepsilon_0}{s}}\norm{x_{_\Omega}^\star-x_{_\Omega}^0}_{2}}\\
    &= \gamma\sqrt{k\cdot {\underset{i \in \mathbb{N}_{+}} \sup\left\{\sigma_{r,i}^2\right\}}},\quad \text{for}\quad k \gg \frac{2 \gamma^2 \varepsilon_0 \norm{x_{_\Omega}^\star-x_{_\Omega}^0}_{2}^2}{9s\cdot{\underset{i \in \mathbb{N}_{+}} \sup\left\{\sigma_{r,i}^2\right\}}}.
\end{align}
with probability (approximately) at least $p^k\left(1-4\exp(-\gamma^2/2)\right)$, where $x^\star$ is any solution of~\eqref{gen.prob.1},  $M_{\nabla g} = 1$, $s \leq 1/L$ for an absolute gradient error, and $M_{\nabla g} = \underset{i \in \mathbb{N}_{+}} \sup\bigg\{\norm{\nabla g(x^i)}_{\infty}\bigg\}$, $s_k := s \leq 1/(L+\delta)$ for a relative gradient error. $\mathbb{E}[.]$ stands for the expectation operator.
\end{theorem}
\begin{proof}
From our previous work~\cite{hamadouche2022sharper}, we have the following convergence error bound
\begin{equation}
   f(x_{_\Omega}^{k+1})-f(x^\star) \leq \frac{1}{k} \bigg[S_{\epsilon_{h_\Omega}}+S_{r_{_\Omega}}+S_{\epsilon_{g_\Omega}} +\frac{1}{2s}\norm{x^\star-x^0}_{2}^2\bigg]
\end{equation}
Defining $\mu_{\cdot,i} = \sigma_{\cdot,i}^2/M_\cdot \leq \mu_\cdot = {\underset{i \in \mathbb{N}_{+}} \sup\left\{\sigma_{\cdot,i}^2\right\}}/M_\cdot$, where $\sigma_{\cdot,i}^2$ and $M_\cdot$ are the variance and upper bound for each error term. By substituting in the Bernstein's concentration inequality~\eqref{Lemma:BernsteinBound} we obtain 
\begin{align} 
    \text{Pr}\bigg(|S - \mathbb{E}\big[S\big]|&\geq t\bigg)\leq 2\textup{exp}\bigg(\frac{-t^2}{2M\sum_{i=1}^k \mu_i + \frac{2}{3}Mt}\bigg),\\
    \text{Pr}\bigg(|S - \mathbb{E}\big[S\big]|&\geq t\bigg)\leq 2\textup{exp}\bigg(\frac{-t^2}{2M(\mu k + \frac{1}{3}t)}\bigg).
\end{align}

When $k \gg t/3\mu$ we have
\begin{align} 
    \text{Pr}\bigg(|S - \mathbb{E}\big[S\big]|&\geq t\bigg)\lessapprox 2\textup{exp}\bigg(-\frac{t^2}{2M\mu k}\bigg).
\end{align}

Setting $t=\gamma\sqrt{M\mu k}$, for $k \gg \gamma^2 M/9\mu$, we obtain
\begin{align} 
    \text{Pr}\bigg(|S - \mathbb{E}\big[S\big]|&\geq \gamma\sqrt{M\mu k}\bigg)\lessapprox 2\textup{exp}\bigg(-\frac{\gamma^2}{2}\bigg).
\end{align}
substituting for $S_{\epsilon_{h_\Omega}}$, $S_{r_{_\Omega}}$ and $S_{\epsilon_{g_\Omega}}$ with their corresponding parameters completes the proof for Theorem~\ref{Theorem3.2}.
\end{proof}
\begin{theorem}[\textbf{Non-asymptotic bounds for AxPGD under approximation errors with rare extreme events}]
\label{Theorem3.3}
With the same assumptions and error models of Theorem~\ref{Theorem3.2}, we have
\begin{equation}
   f(x_{_\Omega}^{k+1})-f(x^\star) \leq \frac{1}{k} \bigg[S_{\epsilon_{h_\Omega}}+S_{r_{_\Omega}}+S_{\epsilon_{g_\Omega}} +\frac{1}{2s}\norm{x^\star-x^0}_{2}^2\bigg]
\end{equation}
where 
\begin{align}
    S_{\epsilon_{h_\Omega}}&=\mathbb{E}\big[\sum_{i=1}^{k}\epsilon_{h_\Omega}^{i}\big]+\gamma\frac{\varepsilon_0}{3}\\
    &=k\mathbb{E}\big[\epsilon_{h_\Omega}\big]+\gamma\frac{\varepsilon_0-\mathbb{E}\big[\epsilon_{h_\Omega}\big]}{3},\quad \text{for} \quad k \ll \frac{\gamma (\varepsilon_0-\mathbb{E}\big[\epsilon_{h_\Omega}\big])^2}{9{\underset{i \in \mathbb{N}_{+}} \sup\left\{\sigma_i^2\right\}}},\\
    S_{\epsilon_{g_\Omega}}&=\mathbb{E}\big[\sum_{i=0}^k {\epsilon_{g_\Omega}^{i}}^\top (x^\star - x_{_\Omega}^{i+1})\big]+\gamma\frac{\sqrt{n}\delta M_{\nabla g}\norm{x_{_\Omega}^\star-x_{_\Omega}^0}_{2}}{3},\\
    &\approx\gamma\frac{\sqrt{n}\delta M_{\nabla g}\norm{x_{_\Omega}^\star-x_{_\Omega}^0}_{2}}{3},\quad \text{for} \quad k \ll \frac{\gamma n(\delta M_{\nabla g}\norm{x_{_\Omega}^\star-x_{_\Omega}^0}_{2})^2}{9{\underset{i \in \mathbb{N}_{+}} \sup\left\{\sigma_i^2\right\}}},\\
    S_{r_{_\Omega}}&=\mathbb{E}\big[-\sum_{i=0}^k \frac{1}{s}{r_{_\Omega}^{i+1}}^\top (x^\star - x_{_\Omega}^{i+1})\big]+\gamma\sqrt{\frac{2\epsilon_h^k}{s}}\frac{\norm{x_{_\Omega}^\star-x_{_\Omega}^0}_{2}}{3},\\
    &\lessapprox\gamma\sqrt{\frac{2\varepsilon_0}{s}}\frac{\norm{x_{_\Omega}^\star-x_{_\Omega}^0}_{2}}{3},\quad \text{for} \quad k \ll \frac{2 \gamma \varepsilon_0 \norm{x_{_\Omega}^\star-x_{_\Omega}^0}_{2}^2}{9s\cdot{\underset{i \in \mathbb{N}_{+}} \sup\left\{\sigma_i^2\right\}}}.
\end{align}
with probability (approximately) at least $p^k(1-4\exp(-\gamma/2))$, where $x^\star$ is any solution of~\eqref{gen.prob.1},  $M_{\nabla g} = 1$, $s \leq 1/L$ for an absolute gradient error, and $M_{\nabla g} = \underset{i \in \mathbb{N}_{+}} \sup\bigg\{\norm{\nabla g(x^i)}_{\infty}\bigg\}$, $s_k := s \leq 1/(L+\delta)$ for a relative gradient error. $\mathbb{E}[.]$ stands for the expectation operator.
\end{theorem}
\begin{proof}
From our previous work~\cite{hamadouche2022sharper}, we have the following convergence error bound
\begin{equation}
   f(x_{_\Omega}^{k+1})-f(x^\star) \leq \frac{1}{k} \bigg[S_{\epsilon_{h_\Omega}}+S_{r_{_\Omega}}+S_{\epsilon_{g_\Omega}} +\frac{1}{2s}\norm{x^\star-x^0}_{2}^2\bigg]
\end{equation}
Defining $\mu_{\cdot,i} = \sigma_{\cdot,i}^2/M_\cdot \leq \mu_\cdot = {\underset{i \in \mathbb{N}_{+}} \sup\left\{\sigma_{\cdot,i}^2\right\}}/M_\cdot$, where $\sigma_{\cdot,i}^2$ and $M_\cdot$ are the variance and upper bound for each error term. By substituting in the Bernstein's concentration inequality~\eqref{Lemma:BernsteinBound} we obtain 
\begin{align} 
    \text{Pr}\bigg(|S - \mathbb{E}\big[S\big]|&\geq t\bigg)\leq 2\textup{exp}\bigg(\frac{-t^2}{2M\sum_{i=1}^k \mu_i + \frac{2}{3}Mt}\bigg),\\
    \text{Pr}\bigg(|S - \mathbb{E}\big[S\big]|&\geq t\bigg)\leq 2\textup{exp}\bigg(\frac{-t^2}{2M(\mu k + \frac{1}{3}t)}\bigg).
\end{align}
When $k \ll t/3\mu$, we have
\begin{align} 
    \text{Pr}\bigg(|S - \mathbb{E}\big[S\big]|&\geq t\bigg)\lessapprox 2\textup{exp}\bigg(-\frac{3t}{{2}M}\bigg).
\end{align}
Setting $t=\gamma M/3$, then for $k \ll \gamma M/9\mu$, we have
\begin{align} 
    \text{Pr}\bigg(|S - \mathbb{E}\big[S\big]|&\geq \gamma M/3\bigg)\lessapprox 2\textup{exp}\bigg(-\frac{\gamma}{2}\bigg).
\end{align}
substituting for $S_{\epsilon_{h_\Omega}}$, $S_{r_{_\Omega}}$ and $S_{\epsilon_{g_\Omega}}$ with their corresponding parameters completes the proof for Theorem~\ref{Theorem3.3}.
\end{proof}

Let us now extend Theorems~\ref{Theorem3.2} and~\ref{Theorem3.3} to the accelerated case.

\subsection{Convergence bounds for the accelerated approximate PGD}
Consider solving problem~\eqref{gen.prob.1} using~\eqref{AxAPGD} under error models~\ref{ErrModel3} and~\ref{ErrModel4}.
\begin{theorem}[\textbf{Convergence bounds for the accelerated AxPGD under approximation errors with rare extreme events}]
\label{Theorem3.4}
Consider the approximate accelerated proximal-gradient descent in~\eqref{AxAPGD} with constant stepsize $s_k := s$. Assume that the gradient error $\{\epsilon_{g_\Omega}^k\}_{k\geq 1}$ and residual proximal error $\{r_{_\Omega}^{k}\}_{k\geq 1}$ sequences satisfy the properties of Model~\ref{ErrModel3} and Model~\ref{ErrModel4}, respectively. Let the momentum sequence $\beta_k$ be defined as $\beta_k = (\alpha_{k-1} - 1)/\alpha_k$, where $\alpha_k$ satisfies
\begin{itemize}
    \item $\alpha_k \geq 1 \quad \forall \quad k > 0$ and $\alpha_0 = 1$
    \item $\alpha_k^2-\alpha_k = \alpha_{k-1}$
    \item $\{\alpha_k\}_{k=0}^\infty$ is an increasing sequence and proportional to $k$ ($O(k)$)
\end{itemize}
Assume that $\norm{x_{_\Omega}^k-x_{_\Omega}^\star}_{2}^2 \leq \norm{x_{_\Omega}^0-x_{_\Omega}^\star}_{2}^2$ hold with probability $p$, for all $k$. Then, the sequence $\{x^k\}$ produced by ~\eqref{AxAPGD} satisfies
\begin{equation}
   f(x_{_\Omega}^{k+1})-f(x^\star) \leq \frac{1}{\alpha_k^2} \bigg[S_{\epsilon_{h_\Omega}}+S_{r_{_\Omega}}+S_{\epsilon_{g_\Omega}} +\frac{1}{2s}\norm{x^\star-x^0}_{2}^2\bigg]
\end{equation}
where 
\begin{align}
    S_{\epsilon_{h_\Omega}}&=\mathbb{E}\big[\sum_{i=0}^k\alpha_i^2\epsilon_{h_\Omega}^{i}\big]+\gamma \sqrt{\frac{k(k + 1)(2k + 1)}{6}{\underset{i \in \mathbb{N}_{+}} \sup\left\{\sigma_{\epsilon_h}^2\right\}}},\notag\\& \text{for}\quad k \gg \frac{\gamma^2\overline{\alpha}^4 (\varepsilon_0-\mathbb{E}\big[\epsilon_{h_\Omega}\big])^2}{9{\underset{i \in \mathbb{N}_{+}} \sup\left\{\sigma_{\epsilon_h}^2\right\}}},\\
    S_{\epsilon_{g_\Omega}}&=\mathbb{E}\big[\sum_{i=1}^{k}\alpha_i{\epsilon_{g_\Omega}^{i}}^\top u_{_\Omega}^i\big]+\gamma \sqrt{\frac{k(k + 1)}{2}{\underset{i \in \mathbb{N}_{+}} \sup\left\{\sigma_i^2\right\}}},\\
    &\approx\gamma \sqrt{\frac{k(k + 1)}{2}{\underset{i \in \mathbb{N}_{+}} \sup\left\{\sigma_{\epsilon_g,i}^2\right\}}},\quad \text{for}\quad k \gg \frac{\gamma^2 \overline{\alpha}^2n(\delta M_{\nabla g}\norm{x_{_\Omega}^\star-x_{_\Omega}^0}_{2})^2}{9{\underset{i \in \mathbb{N}_{+}} \sup\left\{\sigma_{\epsilon_g,i}^2\right\}}},\\
    S_{r_{_\Omega}}&=\mathbb{E}\big[\sum_{i=1}^{k}-\alpha_i\frac{1}{s} {r_{_\Omega}^{i+1}}^\top u_{_\Omega}^i\big]+\gamma \sqrt{\frac{k(k + 1)}{2}{\underset{i \in \mathbb{N}_{+}} \sup\left\{\sigma_{r,i}^2\right\}}},\\
    &\approx\gamma \sqrt{\frac{k(k + 1)}{2}{\underset{i \in \mathbb{N}_{+}} \sup\left\{\sigma_{r,i}^2\right\}}},\quad \text{for}\quad k \gg \frac{2 \gamma^2 \overline{\alpha}^2\varepsilon_0 \norm{x_{_\Omega}^\star-x_{_\Omega}^0}_{2}^2}{9s\cdot{\underset{i \in \mathbb{N}_{+}} \sup\left\{\sigma_{r,i}^2\right\}}}.
    \\
\end{align}
with probability (approximately) at least $p^k\left(1-4\exp(-\gamma^2/2)\right)$, where $\overline{\alpha} = \max\{\alpha_i\}$, $x^\star$ is any solution of~\eqref{gen.prob.1},  $M_{\nabla g} = 1$, $s \leq 1/L$ for an absolute gradient approximation error, and $M_{\nabla g} = \underset{i \in \mathbb{N}_{+}} \sup\bigg\{\norm{\nabla g(x^i)}_{\infty}\bigg\}$, $s_k := s \leq 1/(L+\delta)$ for a relative gradient error. $\mathbb{E}[.]$ stands for the expectation operator.
\end{theorem}
\begin{proof}
From our previous work~\cite{hamadouche2022sharper}, we have the following convergence error bound
\begin{equation}
   f(x_{_\Omega}^{k+1})-f(x^\star) \leq \frac{1}{\alpha_k^2} \bigg[S_{\epsilon_{h_\Omega}}+S_{r_{_\Omega}}+S_{\epsilon_{g_\Omega}} +\frac{1}{2s}\norm{x^\star-x^0}_{2}^2\bigg]
\end{equation}
Defining $\mu_{\cdot,i} = \sigma_{\cdot,i}^2/M_\cdot \leq \mu_\cdot = {\underset{i \in \mathbb{N}_{+}} \sup\left\{\sigma_{\cdot,i}^2\right\}}/M_\cdot$, where $\sigma_{\cdot,i}^2$ and $M_\cdot$ are the variance and upper bound for each error term. By substituting in the Bernstein's concentration inequality~\eqref{Lemma:BernsteinBound} we obtain 
\begin{align}
    S_{\epsilon_{h_\Omega}}&=\mathbb{E}\big[\sum_{i=0}^{k}\alpha_{i}^2\epsilon_{h_\Omega}^{i}\big]+\gamma \sqrt{\sum_{i=0}^{k}\alpha_{i}^2(\varepsilon_0-\mathbb{E}\big[\epsilon_{h_\Omega}\big])\mu_1},\\
    S_{\epsilon_{g_\Omega}}&=\mathbb{E}\big[\sum_{i=1}^{k}\alpha_i{\epsilon_{g_\Omega}^{i}}^\top u_{_\Omega}^i\big]+\gamma \sqrt{\sum_{i=0}^{k}\alpha_{i}\sqrt{n}|\delta|M_u M_{\nabla g}\norm{x^0-x^\star}_2\mu_2},\\
    &\approx\gamma \sqrt{\sum_{i=0}^{k}\alpha_{i}\sqrt{n}|\delta|M_u M_{\nabla g}\norm{x^0-x^\star}_2\mu_2},\\
    S_{r_{_\Omega}}&=\mathbb{E}\big[\sum_{i=1}^{k}-\alpha_i\frac{1}{s} {r_{_\Omega}^{i+1}}^\top u_{_\Omega}^i\big]+\gamma \sqrt{\sum_{i=0}^{k}\alpha_{i}\sqrt{2\varepsilon_0/s}M_u\norm{x^0-x^\star}_2\mu_3},\\
    &\approx\gamma \sqrt{\sum_{i=0}^{k}\alpha_{i}\sqrt{2\varepsilon_0/s}M_u\norm{x^0-x^\star}_2\mu_3}.
    \\
\end{align}
By assumption, we have $\sum_{i=0}^{k}\alpha_{i} \leq \sum_{i=0}^{k}i$ and $\sum_{i=0}^{k}\alpha_{i}^2 \leq \sum_{i=0}^{k}{i}^2$,
substituting for $\sum_{i=1}^{k} i$ by
\begin{equation}
    \sum_{i=1}^{k} i = \frac{k(k + 1)}{2},
\end{equation}
and substituting for $\sum_{i=1}^{k} i^{2}$ by
\begin{equation}
    \sum_{i=1}^{k} i^{2} = \frac{k(k + 1)(2k + 1)}{6}.
\end{equation}
Finally substituting for $\mu_1$, $\mu_2$ and $\mu_3$ completes the proof.
\end{proof}
\begin{theorem}[\textbf{Non-asymptotic bounds for the accelerated AxPGD under approximation errors with rare extreme events}]
\label{Theorem3.5}
With the same assumptions and error models of Theorem~\ref{Theorem3.4}, we have
\begin{equation}
   f(x_{_\Omega}^{k+1})-f(x^\star) \leq \frac{1}{\alpha_k^2} \bigg[S_{\epsilon_{h_\Omega}}+S_{r_{_\Omega}}+S_{\epsilon_{g_\Omega}} +\frac{1}{2s}\norm{x^\star-x^0}_{2}^2\bigg]
\end{equation}
where 
\begin{align}
    S_{\epsilon_{h_\Omega}}&=\mathbb{E}\left[\sum_{i=0}^{k}\alpha_{i}^2\epsilon_{h_\Omega}^{i}\right]+\gamma \frac{{k(k + 1)(2k + 1)}(\varepsilon_0-\mathbb{E}\big[\epsilon_{h_\Omega}\big])}{18},\\
    &\quad \text{for} \quad k \ll \frac{\gamma \overline{\alpha}^4(\varepsilon_0-\mathbb{E}\big[\epsilon_{h_\Omega}\big])^2}{9{\underset{i \in \mathbb{N}_{+}} \sup\left\{\sigma_i^2\right\}}},\notag\\
    S_{\epsilon_{g_\Omega}}&=\mathbb{E}\big[\sum_{i=1}^{k}\alpha_i{\epsilon_g^{i}}^\top u_{_\Omega}^i\big]+\gamma \frac{{k(k + 1)}\sqrt{n}|\delta| M_{\nabla g}M_u\norm{x^0-x^\star}_2}{6},\\
    &\approx\gamma \frac{{k(k + 1)}\sqrt{n}|\delta| M_{\nabla g}M_u\norm{x^0-x^\star}_2}{6},\quad \text{for} \quad k \ll \frac{\gamma \overline{\alpha}^2n(\delta M_{\nabla g}\norm{x_{_\Omega}^\star-x_{_\Omega}^0}_{2})^2}{9{\underset{i \in \mathbb{N}_{+}} \sup\left\{\sigma_i^2\right\}}},\\
    S_{r_{_\Omega}}&=\mathbb{E}\big[\sum_{i=1}^{k}-\alpha_i\frac{1}{s} {r^{i+1}}^\top u_{_\Omega}^i\big]+\gamma \frac{{k(k + 1)}\sqrt{2\varepsilon_0/s}M_u\norm{x^0-x^\star}_2}{6},\\
    &\approx\gamma \frac{{k(k + 1)}\sqrt{2\varepsilon_0/s}M_u\norm{x^0-x^\star}_2}{6},\quad \text{for} \quad k \ll \frac{2 \gamma \overline{\alpha}^2\varepsilon_0 \norm{x_{_\Omega}^\star-x_{_\Omega}^0}_{2}^2}{9s\cdot{\underset{i \in \mathbb{N}_{+}} \sup\left\{\sigma_i^2\right\}}}.
\end{align}
with probability (approximately) at least $p^k\left(1-4\exp(-\gamma^2/2)\right)$, where $\overline{\alpha} = \max\{\alpha_i\}$, $x^\star$ is any solution of~\eqref{gen.prob.1},  $M_{\nabla g} = 1$, $s \leq 1/L$ for an absolute gradient approximation error, and $M_{\nabla g} = \underset{i \in \mathbb{N}_{+}} \sup\bigg\{\norm{\nabla g(x^i)}_{\infty}\bigg\}$, $s_k := s \leq 1/(L+\delta)$ for a relative gradient error. $\mathbb{E}[.]$ stands for the expectation operator.
\end{theorem}

\subsection{Convergence error bounds for the approximate WLM-ADMM}
Let us now consider solving~\eqref{gen.prob.1} using the approximate AxWLM-ADMM~\eqref{AxWLM-ADMMProx} with error model~\ref{ErrModel4}.

\begin{theorem}[\textbf{Convergence bounds for AxWLM-ADMM under approximation errors with rare extreme events}]
\label{Theorem3.6}
Assume \textbf{P.2}, \textbf{P.3}, \textbf{M.4.1-5}. Assume that $\norm{x_{_\Omega}^k-x_{_\Omega}^\star}_{2}^2 \leq \norm{x_{_\Omega}^0-x_{_\Omega}^\star}_{2}^2$ and $\norm{z_{_\Omega}^k-z_{_\Omega}^\star}_{2}^2 \leq \norm{z_{_\Omega}^0-z_{_\Omega}^\star}_{2}^2$ hold with probability $p$, for all $k$. Then, with probability at least $p^k\left(1 - 4\exp(-\frac{\gamma^2}{2})\right)$, for any $\gamma >0$, the sequence generated using the AxWLM-ADMM scheme~\eqref{WLM-ADMMProx} with $\lambda_x=\lambda_z=1$ and a fixed PSD matrix $L$ satisfies
\begin{align}
    &\frac{1}{k+1}\sum_{i=0}^{k}f(x_{_\Omega}^{i+1},z_{_\Omega}^{i+1})-f(x^\star,z^\star)
    +\frac{1}{k+1}\sum_{i=0}^{k}\langle\frac{1}{\lambda}L  u_{_\Omega}^{i+1},v_{_\Omega}^{k+1}-v_{_\Omega}^k\rangle\leq\notag\\
    &\frac{1}{2(k+1)}\left[\|x^0-x^\star\|^2_{ M_x }+\|z^0-z^\star\|^2\right]
    +\frac{1}{k+1}\left[S_{\epsilon_{g_\Omega}}+S_{\epsilon_{h_\Omega}}+D_{r_g,x}+D_{r_h,z}\right],
\end{align}
where 
\begin{align}
    S_{\epsilon_{g_\Omega}}&=k \mathbb{E}\big[\epsilon_{g_\Omega}\big]+\gamma\sqrt{k\cdot \underset{i \in \mathbb{N}_{+}} \sup\left\{\sigma_{\epsilon_g,i}^2\right\}},\quad \text{for}\quad k \gg \frac{\gamma^2 (\varepsilon_{g_0}-\mathbb{E}\big[\epsilon_{g_\Omega}\big])^2}{9{\underset{i \in \mathbb{N}_{+}} \sup\left\{\sigma_{\epsilon_g,i}^2\right\}}},\\
    S_{\epsilon_{h_\Omega}}&=k \mathbb{E}\big[\epsilon_{h_\Omega}\big]+\gamma\sqrt{k\cdot \underset{i \in \mathbb{N}_{+}} \sup\left\{\sigma_{\epsilon_h}^2\right\}},\quad \text{for}\quad k \gg \frac{\gamma^2 (\varepsilon_{h_0}-\mathbb{E}\big[\epsilon_{h_\Omega}\big])^2}{9{\underset{i \in \mathbb{N}_{+}} \sup\left\{\sigma_{\epsilon_h}^2\right\}}},\\
    D_{r_g,x}&=\sqrt{2\lambda_{\max} ( M_x^\top M_x )\varepsilon_g^{k+1}}\|x_{_\Omega}^{k+1}-x^\star\|\Big]\leq \sqrt{2\lambda_{\max} ( M_x^\top M_x )\varepsilon_{g_0}}\|x^{0}-x^\star\|\Big]\\
    D_{r_h,z}&=\sqrt{2\varepsilon_h^{k+1}}\|z_{_\Omega}^{k+1}-z^\star\|\leq \sqrt{2\varepsilon_{h_0}}\|z^{0}-z^\star\|
\end{align}
with probability (approximately) at least $p^k\left(1-4\exp(-\gamma^2/2)\right)$, where $(x^\star,z^\star)$ is any solution of~\eqref{Eq:Problem} and $\mathbb{E}[.]$ stands for the expectation operator.
\end{theorem}
\begin{proof}
From our previous work~\cite{hamadouche2022probabilistic}, we have the following convergence error bound
\begin{align}
    &\frac{1}{k+1}\sum_{i=0}^{k}f(x^{i+1},z^{i+1})-f(x^\star,z^\star)
    +\frac{1}{k+1}\sum_{i=0}^{k}\langle\frac{1}{\lambda}L  u^{i+1},v^{k+1}-v^k\rangle\leq\notag\\
    &\frac{1}{2(k+1)}\left[\|x^0-x^\star\|^2_{ M_x }+\|z^0-z^\star\|^2\right]
    +\frac{1}{k+1}\left[S_{\epsilon_{g_\Omega}}+S_{\epsilon_{h_\Omega}}+D_{r_g,x}+D_{r_h,z}\right],
\end{align}
Defining $\mu_{\cdot,i} = \sigma_{\cdot,i}^2/M_\cdot \leq \mu_\cdot = {\underset{i \in \mathbb{N}_{+}} \sup\left\{\sigma_{\cdot,i}^2\right\}}/M_\cdot$, where $\sigma_{\cdot,i}^2$ and $M_\cdot$ are the variance and upper bound for each error term. By substituting in the Bernstein's concentration inequality~\eqref{Lemma:BernsteinBound} we obtain 
\begin{align} 
    \text{Pr}\bigg(|S - \mathbb{E}\big[S\big]|&\geq t\bigg)\leq 2\textup{exp}\bigg(\frac{-t^2}{2M\sum_{i=1}^k \mu_i + \frac{2}{3}Mt}\bigg),\\
    \text{Pr}\bigg(|S - \mathbb{E}\big[S\big]|&\geq t\bigg)\leq 2\textup{exp}\bigg(\frac{-t^2}{2M(\mu k + \frac{1}{3}t)}\bigg).
\end{align}

When $k \gg t/3\mu$ we have
\begin{align} 
    \text{Pr}\bigg(|S - \mathbb{E}\big[S\big]|&\geq t\bigg)\lessapprox 2\textup{exp}\bigg(-\frac{t^2}{2M\mu k}\bigg).
\end{align}

Setting $t=\gamma\sqrt{M\mu k}$, for $k \gg \gamma^2 M/9\mu$, we obtain
\begin{align} 
    \text{Pr}\bigg(|S - \mathbb{E}\big[S\big]|&\geq \gamma\sqrt{M\mu k}\bigg)\lessapprox 2\textup{exp}\bigg(-\frac{\gamma^2}{2}\bigg).
\end{align}
substituting $S_{\epsilon_{g_\Omega}} = \sum_{i=0}^{k}\epsilon_{g_\Omega}^{i+1}$ and $S_{\epsilon_{h_\Omega}} = \sum_{i=0}^{k}\epsilon_{h_\Omega}^{i+1}$ with their corresponding parameters completes the proof.
\end{proof}

\begin{theorem}[\textbf{Non-asymptotic bounds for AxWLM-ADMM under approximation errors with rare extreme events}]
\label{Theorem3.7}
Assume that $\norm{x_{_\Omega}^k-x_{_\Omega}^\star}_{2}^2 \leq \norm{x_{_\Omega}^0-x_{_\Omega}^\star}_{2}^2$ and $\norm{z_{_\Omega}^k-z_{_\Omega}^\star}_{2}^2 \leq \norm{z_{_\Omega}^0-z_{_\Omega}^\star}_{2}^2$ hold with probability $p$, for all $k$. Then, with probability at least $p^k\left(1 - 4\exp(-\frac{\gamma^2}{2})\right)$, for any $\gamma >0$, the sequence generated using the AxWLM-ADMM scheme~\eqref{WLM-ADMMProx} with $\lambda_x=\lambda_z=1$ and a fixed PSD matrix $L$ satisfies
\begin{align}
    &\frac{1}{k+1}\sum_{i=0}^{k}f(x_{_\Omega}^{i+1},z_{_\Omega}^{i+1})-f(x^\star,z^\star)
    +\frac{1}{k+1}\sum_{i=0}^{k}\langle\frac{1}{\lambda}L  u_{_\Omega}^{i+1},v_{_\Omega}^{k+1}-v_{_\Omega}^k\rangle\leq\notag\\
    &\frac{1}{2(k+1)}\left[\|x^0-x^\star\|^2_{ M_x }+\|z^0-z^\star\|^2\right]
    +\frac{1}{k+1}\left[S_{\epsilon_{g_\Omega}}+S_{\epsilon_{h_\Omega}}+D_{r_g,x}+D_{r_h,z}\right],
\end{align}
where 
\begin{align}
    S_{\epsilon_{g_\Omega}}&=k\mathbb{E}\big[\epsilon_{g_\Omega}\big]+\gamma\frac{\varepsilon_{g_0}-\mathbb{E}\big[\epsilon_{g_\Omega}\big]}{3},\quad \text{for} \quad k \ll \frac{\gamma (\varepsilon_{g_0}-\mathbb{E}\big[\epsilon_{g_\Omega}\big])^2}{9{\underset{i \in \mathbb{N}_{+}} \sup\left\{\sigma_i^2\right\}}},\\
    S_{\epsilon_{h_\Omega}}&=k\mathbb{E}\big[\epsilon_{h_\Omega}\big]+\gamma\frac{\varepsilon_{h_0}-\mathbb{E}\big[\epsilon_{h_\Omega}\big]}{3},\quad \text{for} \quad k \ll \frac{\gamma (\varepsilon_{h_0}-\mathbb{E}\big[\epsilon_{h_\Omega}\big])^2}{9{\underset{i \in \mathbb{N}_{+}} \sup\left\{\sigma_i^2\right\}}},\\
    D_{r_g,x}&=\sqrt{2\lambda_{\max} ( M_x^\top M_x )\varepsilon_g^{k+1}}\|x_{_\Omega}^{k+1}-x^\star\|\Big]\leq \sqrt{2\lambda_{\max} ( M_x^\top M_x )\varepsilon_{g_0}}\|x^{0}-x^\star\|\Big]\\
    D_{r_h,z}&=\sqrt{2\varepsilon_h^{k+1}}\|z_{_\Omega}^{k+1}-z^\star\|\leq \sqrt{2\varepsilon_{h_0}}\|z^{0}-z^\star\|
\end{align}
with probability (approximately) at least $p^k\left(1-4\exp(-\gamma/2)\right)$, where $(x^\star,z^\star)$ is any solution of~\eqref{Eq:Problem} and $\mathbb{E}[.]$ stands for the expectation operator.
\end{theorem}
\begin{proof}
From our previous work~\cite{hamadouche2022probabilistic}, we have the following convergence error bound
\begin{align}
    &\frac{1}{k+1}\sum_{i=0}^{k}f(x_{_\Omega}^{i+1},z_{_\Omega}^{i+1})-f(x^\star,z^\star)
    +\frac{1}{k+1}\sum_{i=0}^{k}\langle\frac{1}{\lambda}L  u_{_\Omega}^{i+1},v_{_\Omega}^{k+1}-v_{_\Omega}^k\rangle\leq\notag\\
    &\frac{1}{2(k+1)}\left[\|x^0-x^\star\|^2_{ M_x }+\|z^0-z^\star\|^2\right]
    +\frac{1}{k+1}\left[S_{\epsilon_{g_\Omega}}+S_{\epsilon_{h_\Omega}}+D_{r_g,x}+D_{r_h,z}\right],
\end{align}
Defining $\mu_{\cdot,i} = \sigma_{\cdot,i}^2/M_\cdot \leq \mu_\cdot = {\underset{i \in \mathbb{N}_{+}} \sup\left\{\sigma_{\cdot,i}^2\right\}}/M_\cdot$, where $\sigma_{\cdot,i}^2$ and $M_\cdot$ are the variance and upper bound for each error term. By substituting in the Bernstein's concentration inequality~\eqref{Lemma:BernsteinBound} we obtain 
\begin{align} 
    \text{Pr}\bigg(|S - \mathbb{E}\big[S\big]|&\geq t\bigg)\leq 2\textup{exp}\bigg(\frac{-t^2}{2M\sum_{i=1}^k \mu_i + \frac{2}{3}Mt}\bigg),\\
    \text{Pr}\bigg(|S - \mathbb{E}\big[S\big]|&\geq t\bigg)\leq 2\textup{exp}\bigg(\frac{-t^2}{2M(\mu k + \frac{1}{3}t)}\bigg).
\end{align}
When $k \ll t/3\mu$, we have
\begin{align} 
    \text{Pr}\bigg(|S - \mathbb{E}\big[S\big]|&\geq t\bigg)\lessapprox 2\textup{exp}\bigg(-\frac{3t}{{2}M}\bigg).
\end{align}
Setting $t=\gamma M/3$, then for $k \ll \gamma M/9\mu$, we have
\begin{align} 
    \text{Pr}\bigg(|S - \mathbb{E}\big[S\big]|&\geq \gamma M/3\bigg)\lessapprox 2\textup{exp}\bigg(-\frac{\gamma}{2}\bigg).
\end{align}
substituting $S_{\epsilon_{g_\Omega}} = \sum_{i=0}^{k}\epsilon_{g_\Omega}^{i+1}$ and $S_{\epsilon_{h_\Omega}} = \sum_{i=0}^{k}\epsilon_{h_\Omega}^{i+1}$ with their corresponding parameters completes the proof.
\end{proof}
The bounds of Theorem~\ref{Theorem3.3}, Theorem~\ref{Theorem3.5} and Theorem~\ref{Theorem3.7} ($O(1/k)$) are sharper than the bounds of~\cite{hamadouche2022sharper} and~\cite{hamadouche2022probabilistic} ($O(1/\sqrt{k})$), respectively. Moreover, since we have $\sigma_\cdot^2 \ll {\varepsilon_{\cdot_0}}$ for errors with extreme events, Theorem~\ref{Theorem3.2}, Theorem~\ref{Theorem3.4} and Theorem~\ref{Theorem3.6} are asymptotically much sharper than the best probabilistic bounds so far (see \cite{hamadouche2021sspd, hamadouche2022probabilistic}) since the slow error terms (i.e., $O(1/\sqrt{k})$) only depend on the standard deviation $\sigma_.$ rather than the maximum range of errors $\varepsilon_{0_.}$ as in~\eqref{Eq:2.3} and \eqref{Eq:2.4}. 

\section{EXPERIMENTAL RESULTS}
\label{Section: Experimental Results}
We validate the proposed approach by applying Theorems~\ref{Theorem3.2} and \ref{Theorem3.3} to a spacecraft MPC problem~\cite{hegrenaes2005spacecraft} and compare the sharpness of the combined bound (referred to as "Thrm\_6\&7" below) against Theorem 2 of~\cite{hamadouche2022sharper} for one iteration of MPC. We allow $5000$ iterations for the AxPGD solver to emphasize the asymptotic behaviour. 
The minimum total variation or minimum oscillatory\footnote{Also, minimum bandwidth.} model predictive control (MPC) problem is formulated in terms of the differential control ($\Delta u$) with the following augmented state space description\footnote{Seven states are considered here:  Roll, Pitch, Yaw, $\omega_1$, $\omega_2$, $\omega_3$, $\omega_w$, where Roll, Pitch, Yaw describe the rotating angles of the body frame relative to the orbit  frame, and $\omega_1$, $\omega_2$, $\omega_3$ are the corresponding angular velocities. $\omega_w$ is the angular velocity along the spin axis. The thrusters are controlled by three input voltages, $\tau_1$, $\tau_2$, $\tau_3$, and the reaction wheel is controlled by input voltage $\tau_w$ accordingly.}
\begin{equation}
\begin{split}
    \begin{bmatrix}\Delta x[k+1] \\ y(k+1)\end{bmatrix} &= \begin{bmatrix} A_d & O_m\\C A_d & I_m \end{bmatrix} \begin{bmatrix}\Delta x[k] \\ y(k)\end{bmatrix}+\begin{bmatrix} B_d \\ C B_d \end{bmatrix} \Delta u[k]\\
    y[k] &= \begin{bmatrix} O_n & I_m \end{bmatrix} \begin{bmatrix}\Delta x[k] \\ y(k)\end{bmatrix}\\
\end{split}
\label{augmented}
\end{equation}
where
\begin{equation}
    (A_d, B_d,C,D) = \left(e^{Ah}, \int_0^h e^{At}B dt,C,D\right),
\end{equation}
is the discretized state space description that is obtained from zero-order-hold (ZOH) sampling of the continuous state space description $(A, B, C, D)$ of~\cite{hegrenaes2005spacecraft}. Define the new state space description of the augmented system~\eqref{augmented} as
\begin{equation}
    X[k] = \begin{bmatrix}\Delta x[k] \\ y(k)\end{bmatrix},\quad
    A_a = \begin{bmatrix} A_d & O_m\\C A_d & I_m \end{bmatrix},\quad 
    B_a = \begin{bmatrix} B_d \\ C B_d \end{bmatrix}.
\end{equation}
Then MPC is the solution to the following composite optimization problem
\begin{equation}
    \begin{split}
            \min_{\Delta U \in\mathbb{R}^{p\times N_c}} F(\Delta U):=g(\Delta U)+h(\Delta U),
    \end{split}
    \label{Eq:MPCLASSO}
\end{equation}
where $g$ and $h$ are given by
\begin{equation}
    \begin{split}
            g(\Delta U) &:= \norm{\big(\Phi^\top Q\Phi 
            + R\big)^{\frac{1}{2}}\Delta U -\big(\Phi^\top Q\Phi + R\big)^{-\frac{1}{2}}\Phi^\top Q\big(R_s - \Psi x(k)\big)}_2^2;\\
            h(\Delta U) &:= \lambda \norm{\Delta U}_1.
            \label{(g(U),h(U))}
    \end{split}
\end{equation}
$\Phi_a$, $Y$ and $F_a$ are defined as
\begin{align}
    Y &= \begin{bmatrix}y[k+1|k]\\\vdots\\y[k+N_p|k]\end{bmatrix},\quad F_a = \begin{bmatrix}CA_a\\\vdots\\CA_a^{N_p}\end{bmatrix}, \\
    \Phi_a &= \begin{bmatrix}
    CB_a& \dots& 0\\
    CA_aB_a& CB_a& \dots& 0\\
    CA_a^2B_a& CA_aB_a& \dots& 0\\
    \vdots& \ddots& \ddots& \vdots\\
    CA_a^{N_p-1}B_a& CA_a^{N_p-2}B_a& \dots& CA_a^{N_p-N_c}B_a\end{bmatrix},
\end{align}
where $N_p$ and $N_c$ are the prediction and control horizons, respectively. 

The following plot is obtained from solving the MPC problem~\eqref{Eq:MPCLASSO} using the AxPGD solver~\eqref{AxPGD} with injected truncated random errors that satisfy the following inequality
\begin{equation}
     28\cdot\sigma_{\cdot}^2 \leq 2M_{\epsilon_{\cdot}},
\end{equation}
where $2M_{\epsilon_{\cdot}}$ is the numerical range of error $\epsilon_{\cdot}$ and $\sigma_{\cdot}$ is the corresponding standard deviation. $\epsilon_{\cdot}$ stands for either $\epsilon_{g}$ or $\epsilon_{h}$.

\begin{figure}[h]
\centering
\captionsetup{justification=centering}
\includegraphics[width=10cm]{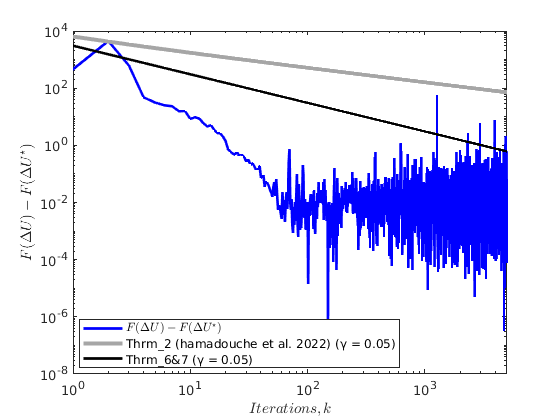}
\caption{Upper bounds based on Theorems \ref{Theorem3.2} and \ref{Theorem3.3} vs Theorem 2~\eqref{Eq:2.3} of Hamadouche et al. 2022~\cite{hamadouche2022sharper} (with $\gamma=0.05$; $\delta = 2.2\times 10^{1}; \epsilon_0 = 10^{1}$).}
\label{fig:fig1}

\end{figure}
From~\autoref{fig:fig1}, we can clearly see that the proposed bound is sharper than the bound of Theorem 2~\cite{hamadouche2022sharper}. Consequently, the new bound of Theorems \ref{Theorem3.2} and \ref{Theorem3.3} explains the asymptotic as well as the non-asymptotic behaviours of the approximate proximal-gradient descnet algorithm (AxPGD)~\eqref{AxPGD} better than any previously derived bounds.

\section{CONCLUSIONS}
\label{Section: Conclusions}
In this work, we improved our previous results~\cite{hamadouche2021sspd, hamadouche2022sharper} and established sharper convergence bounds on the objective function values of approximate proximal-gradient descent (AxPGD), approximate accelerated proximal-gradient descent (AxAPGD) schemes as well as a new approximate generalized proximal ADMM (AxWLM-ADMM) scheme. We exploited second-order moments of approximation errors that manifest rare extreme events and we derived probabilistic asymptotic and non-asymptotic bounds based on Bernstein's concentration inequality which yielded bounds that are explicit functions of the range and variance of approximation errors. As a future work, we will instantiate new algorithms from the AxWLM-ADMM class and compare their performance vs efficiency trade-off using new metrics.
\addcontentsline{toc}{section}{Acknowledgment}
\section*{Acknowledgements}
This work was supported by the Engineering and Physical Research Council (EPSRC) grants (EP/T026111/1 and EP/S000631/1), and the MOD University Defence Research Collaboration (UDRC).
\bibliographystyle{siamplain}
\bibliography{references.bib}
\pdfoutput=1
\end{document}